\documentclass[reqno,11pt]{amsart}
\usepackage{amsthm}
\usepackage{filecontents}\usepackage[margin=1in]{geometry}
\usepackage{amssymb,amsfonts}
\usepackage{mathtools}
\usepackage{nicematrix}
\usepackage{enumerate}
\usepackage{mathrsfs}
\usepackage{mathtools}
\usepackage{xcolor}
\usepackage{hyperref}

\usepackage{blkarray}
\usepackage{amsthm}
\theoremstyle{definition}
\newtheorem{thm}{Theorem}[section]
\theoremstyle{definition}

\newtheorem{cor}[thm]{Corollary}
\newtheorem{prop}[thm]{Proposition}
\newtheorem{lem}[thm]{Lemma}

\newtheorem{quest}[thm]{Question}
\newtheorem{defn}[thm]{Definition}

\newtheorem{rem}{Remark}[section] 
\def\A{{\mathbb A}}
\def\C{{\mathbb C}}
\def\Z{{\mathbb Z}}
\def\D{{\mathcal D}}

\def\Q{{\mathbb Q}}
\def\A{{\mathbb A}}
\def\R{{\mathbb R}}

\def\F{{\mathbb F}}

\def\C{{\mathbb C}}

\DeclareMathOperator{\an}{an}

\DeclareMathOperator{\End}{End}

\DeclareMathOperator{\GL}{GL}

\DeclareMathOperator{\sep}{sep}

\DeclareMathOperator{\Gal}{Gal}

\def\bbG{{\mathbb G}}

\def\R{{\mathbb R}}
\def\F{{\mathbb F}}
\def\F{{\mathbb F}}

\DeclareMathOperator{\gr}{gr}
\DeclareMathOperator{\sups}{ss}
\DeclareMathOperator{\insep}{insep}

\DeclareMathOperator{\univ}{univ}
\DeclareMathOperator{\Sh}{Sh}
\DeclareMathOperator{\Res}{Res}
\DeclareMathOperator{\GSpin}{GSpin}
\DeclareMathOperator{\GSp}{GSp}
\DeclareMathOperator{\rig}{rig}
\DeclareMathOperator{\SO}{SO}
\DeclareMathOperator{\KS}{KS}

\DeclareMathOperator{\im}{im}

\DeclareMathOperator{\Aut}{Aut}

\usepackage{tikz-cd}
\usepackage[all,cmtip]{xy}

\hypersetup{ colorlinks=true,
linkcolor=blue
}

\begin{document}
\title{Monodromy results for abelian surfaces and K3 surfaces with bad reduction}
\author{Tejasi Bhatnagar}

\begin{abstract}
The purpose of this paper is to prove a local $p$-adic monodromy theorem for ordinary abelian surfaces and K3 surfaces with bad reduction in characteristic $p.$ As an application, we get a finiteness result for the reduction of their Hecke orbits in the case of type II supersingular reduction.
\end{abstract}
\maketitle

\tableofcontents

\section{Introduction}
 Throughout this paper $K$ will denote the local function field $\F_q((t))$ in characteristic $p$. Let $R$ be its valuation ring and $k$ its residue field. In this paper, we aim to extend a $p$-adic  monodromy result for elliptic curves defined over $K$ to  orthogonal Shimura varieties. Let $E$ over $K$ be an ordinary elliptic curve. Then we consider the following Galois representation, also known as the monodromy representation associated to the $p$-power torsion $E[p^n](\overline{K}) = \Z/p^n\Z$ of $E$:
$$\rho_{E}:\Gal(K^{\sep}/K)\rightarrow \Aut(\Z_p) = \Z_p^{\times}$$
The central question that we generalise is the following: how do we describe the above $p$-adic Galois representation ? More precisely, we wish to understand the  Galois action and its ramification.

For elliptic curves we have a complete answer that we briefly describe along with relevant work in this direction. When $E$ has good ordinary reduction, then the reduction map on the $p$-power torsion is Galois invariant and an isomorphism. That is, the action of $\Gal(K^{\sep}/K)$ is unramified.  Same is true in higher dimensions as well. The case of good supersingular reduction is more interesting and goes back to the work of Igusa \cite{igusa}. Igusa studied the monodromy representation of the universal elliptic curve around a supersingular point in $\mathcal{A}_1$ and showed that in such a case, the image of the monodromy representation is $\Z_p^{\times}$.
The proof is hands on and shows that, when $E$ has supersingular reduction, then as we attach (coordinates of) $p$-power torsion to the base field $K$, we \textit{eventually} get totally ramified extensions. A different proof using a ``formal group argument" is given by Katz in \cite{katz} that can be generalised to $p$-divisible groups (see Section \ref{section3}). 
In \cite{chaimonodromy}, Chai extends Katz' argument to a one dimensional $p$-divisible group $\mathscr{G}$ over $R$ with generic height one. Building on the previous work of Gross \cite{gross}, he proves analogous results for ramification and the upper breaks of the abelian field extension obtained from the $K^{\sep}$-torsion points of the \'etale quotient of $\mathscr{G}$. As a consequence of these results, the image of the associated one dimensional Galois representation is open in $\Z_p^{\times}$. In higher dimensions, the first result is due to Ekedahl \cite{ekedahl}, utilising deformation theory and   arguments similar to Katz to show that the monodromy of the universal deformation of a product of supersingular elliptic curves is ``big". Much like Igusa's result, we expect the $p$-adic monodromy of a family of abelian varieties to be large, as explained in \cite{chaimethodsmonodromy}. However, not much is known about local monodromy representations. For instance, analogous to Ekedahl's result, Chai asks the following question in \cite{chaimonodromy} which is a direct generalization of Igusa's result in higher dimension.
\begin{quest}(Chai)\label{monodromy}
 Let $A$ over $K$ be an ordinary abelian variety with good supersingular reduction. What can we say about the local monodromy representation associated to its $p$-power torsion points?
 \end{quest}
  
In such a case, we expect the image of inertia to have finite index in the image of the Galois group. The characteristic zero version of this question is one of the main results of \cite{KLSS}. The methods in the paper use the theory of isocrystals over $\Q_p$ which is harder to exploit in the function field setting. This summarises the story when $E$ has good reduction.

As a result of Tate's uniformisation theorem, it is straightforward to understand the Galois action when $E$ has semi-stable reduction. The action turns out to be quite different in that, we only see inseparable extensions of $K,$ once we attach the $p$-power torsion of $E(K)$. We discuss this in detail in Section \ref{badred} along with Katz' proof of Igusa's result in Section \ref{section3}. 
\subsubsection{Main results}In the spirit of the previous results, we prove an analogous monodromy result for abelian surfaces and K3 surfaces with semi-stable reduction. We state our main results below:
\begin{thm}[Monodromy of abelian surfaces]\label{thm1}
Let $A$ be an ordinary abelian surface over $K$ with semi-stable reduction. Denote by $A_0$ its reduction over $k$. Let
$$\rho_{A}:\Gal(K^{\sep}/K)\rightarrow \GL_2(\mathbb{Z}_p)$$ be the monodromy representation associated to the $p^n$-torsion $A[p^n](\overline{K}) = (\Z/p^n\Z)^2$ of $A$. Then the Galois representation is described as follows:
\begin{enumerate}

\item Suppose $A$ has semi-abelian reduction. That is, $A_0$ is an extension of a torus by an elliptic curve over $k$:
$$0\rightarrow \bbG_m\rightarrow A_0\rightarrow B_0\rightarrow 0$$
\begin{enumerate}
    \item If $B_0$ is ordinary, then the action of the inertia subgroup is unipotent.
    \item If $B_0$ is supersingular, then the image of the inertia subgroup has finite index in the entire image of the Galois group.

\end{enumerate}
    \item Suppose $A$ totally degenerates, that is $A_0\simeq \bbG_m^2$, then the Galois group has trivial image. 

\end{enumerate}

\end{thm}

We note that case (1b) is an analogue of Igusa's result. In this case we crucially need the fact that $B_0$ is an elliptic curve that allows us to build on Igusa's proof in dimension one. All the other cases are easier to prove and can be generalised. Extending our strategy to higher dimensions will first require  us to answer Question \ref{monodromy}, stated above. The main ingredient of the proof of Theorem \ref{thm1} involves Raynaud extensions that serve as uniformising spaces for abelian varieties with semi-stable reduction. We elaborate more on this in Section \ref{strategy} and Section \ref{raynauduniformisation}.

\subsubsection{} We prove a more general monodromy result for orthogonal Shimura varieties associated to quadratic lattices of signature $(n,2)$. As a special case $n=3$, recovers Theorem \ref{thm1} for abelian surfaces, while for $n\leq 19$, we get a monodromy theorem for moduli spaces that parametrise K3 surfaces with certain line bundles on it. To state the results for K3 surfaces we introduce some terminology first. Let $X/K$ be an ordinary K3 surface with bad reduction. We consider $H^{2}_{\text{cris}}(X)(-1)$, the Tate twist of its crystalline cohomology. This is an $F$-crystal and carries an action of the ``Frobenius". Our primary object of study is the monodromy representation  associated to the  unit root crystal of the primitive cohomology:  $$\rho_{K3}:\Gal(K^{\sep}/K)\rightarrow \SO_{n}(\Z_p)$$  
 Orthogonal Shimura varieties have two kinds of Baily-Borel boundary components: Zero dimensional and one dimensional components that are isomorphic to the modular curve $\mathcal{A}_1$. If a K3 surface specialises to a zero dimensional boundary component we say it has type III reduction. Whereas  if it specialises to a one dimensional boundary, we say it has type II reduction. Furthermore, if the K3 surface reduces to the ordinary stratum of $\mathcal{A}_1$, we say that the K3 surface has type II ordinary reduction, while if it reduces to the supersingular stratum, we say that it has type II supersingular reduction. 

\begin{thm}[Monodromy of K3 surfaces]\label{2}
Let $X$ be an ordinary K3 surface over $K$ with rank $n$ primitive cohomology associated to $H^2_{\text{cris}}(X)(-1)$.   Write $$\rho_X: \Gal(K^{\sep}/K)\rightarrow \SO_{n}(\Z_p)$$ to be the associated monodromy representation.

\begin{enumerate}
\item Suppose that $X$ degenerates to the one dimensional boundary $\mathcal{A}_1$, that is, it has type II reduction.

\begin{enumerate}
\item If $X$ has type II ordinary reduction then the action of the inertia subgroup is unipotent.
\item If $X$ has type II supersingular reduction, then the image of the inertia subgroup has finite index in the entire image of the Galois group. 
\end{enumerate}
\vspace{1mm}
\item Suppose that $X$ has type III reduction, then the associated Galois representation $\rho_X$ has trivial image.
\end{enumerate}
\end{thm}
\subsubsection{Strategy to prove the monodromy theorems}\label{strategy}
Generalising Tate's theorem to higher dimensions, Raynaud in \cite{raynaud} gives a uniformisation of abelian varieties with semi-stable reduction. This is discussed in detail in Section \ref{raynauduniformisation} of the paper. If $A$ is an abelian variety over $K$ with semi-stable reduction, then $A(\overline{K}) = Z(\overline{K})/M$ where $Z$ is a semi-abelian variety over $K$ and $M$ is a lattice in $Z$.  We call $Z$ the \textit{Raynaud extension} of $A$. Hence we get the following diagram: 

$$\xymatrix{ &M\ar[d]&\\ T\ar[r] & Z \ar[r] \ar[d]&B& \\
& A&  }$$

 Here $B$ is an abelian variety and $T$ is a torus. In particular, when $A$ has dimension two, then either $B=0$ in which case $Z$ is a torus or $B$ is an elliptic curve. Using this information, we write a bases of the $p$-power torsion of $A$ and study Galois action along with its ramification. 

 In order to study the monodromy of K3 surfaces and more generally orthognal Shimura varieties, one of the key tools that we use in this paper is the Kuga-Satake construction. This associates to every K3 surface, an abelian variety (of high dimension). In Section \ref{raynaudks} we prove that the Kuga-Satake abelian variety mirrors the reduction type of its K3 surface. See \cite{kugaSatakedeg} for a similar proof of this fact, however, this paper exploits Hodge theory in a different set up than ours. Furthermore, we also show that the Raynaud extension of the Kuga-Satake abelian variety has a simple description. This is recorded in the following proposition of this paper.

\begin{prop}\label{mainthmred}
Let $X$ be an ordinary K3 surface over $K$ with bad reduction.   Write $\KS(X)$ to be the associated Kuga-Satake abelian variety  with $d = \dim(X).$ 
\begin{enumerate}
\item Suppose $X$ has type II reduction. Then $\KS(X)$ has semi-abelian reduction. Moreover, the Raynaud extension of the Kuga-Satake abelian variety $Z_{\KS}$ over $K$ is an extension
$$0\rightarrow \bbG_m^{d/2}\rightarrow Z_{\KS}\rightarrow B\rightarrow 0$$
where $B$ is isogenous to a product of $d/2$ copies of an elliptic curve $E$ over $K$.
\vspace{1mm}
\item Suppose $X$ has type III reduction. Then the associated Kuga-Satake abelian variety has totally bad reduction. That is, $B=0$ in the above exact sequence.
\end{enumerate}

\end{prop}
Under the Kuga-Satake map, it is enough to study the monodromy of the Kuga-Satake abelian variety. Indeed the Galois representation of the Kuga-Satake abelian variety is a lift of the representation associated to the K3 surface. By Proposition \ref{mainthmred} we know that associated Kuga-Satake abelian variety KS(X) has bad reduction. In particular, we can use the
 Raynaud extension of $\KS(X)$ to understand the associated monodromy. When the Kuga-Satake abelian variety completely degenerates, the result is trivial and follows from the discussion in Section \ref{badred} about monodromy of abelian varieties that totally degenerate.  The interesting case that requires more work is when $\KS(X)$ has semi-abelian reduction. Even though $\KS(X)$ has high dimension, its Raynaud extension $Z_{\KS}$ turns out to be an extension of a torus by a product of $d/2$ copies of an elliptic curve where $d = \dim\KS(X)$. We elaborate more on this fact below. Once we have such a description of $Z_{\KS}$, we similarly write down a bases of the $p$-power torsion of $\KS(X)$ and study the ramification.  The general monodromy result is stated as Theorem \ref{Gspinmainthm} for orthogonal Shimura varieties.

\subsubsection{Raynaud extensions and the philosophy of the problem:}\label{philosophy}
 The data of Raynaud extensions of abelian varieties with bad reduction is parametrized by the boundary of Siegal Shimura varieties. We describe this briefly. The Siegal Shimura variety $\mathcal{A}_g$, that is the moduli space of $g$-dimensional abelian varieties  admits many compactifications. The minimal compactification is called the Baily-Borel compactification. This is however not smooth when $g>1$. The toroidal compactification is smooth and admits a map to the Baily-Borel compactification. We have the following maps between the boundaries:
$$\text{Formal completion along the toroidal boundary $\rightarrow$ Toroidal boundary $\rightarrow$ Baily-Borel boundary }$$
The boundary components carry a variation of mixed Hodge structures. Roughly speaking, the formal completion  parametrises the data of the Raynaud extension $Z$ along with an embedding $\alpha: \Z^r\hookrightarrow Z$ of a lattice.  While the mixed Hodge structure on the toroidal boundary corresponds to the universal Raynaud extension as a semi-abelian scheme over the toroidal boundary. The last map from the toroidal boundary corresponds to taking the quotient of the mixed Hodge structure to get a pure Hodge structure of weight one. Hence the Baily-Borel boundary parametrises the abelian quotient of the Raynaud extension.  Because of the theory of good integral models, this makes sense in characteristic  $p$ as well.
For example consider $\mathcal{A}_2$, the moduli space of abelian surfaces over $\F_p$. Then the Baily-Borel boundary is either a modular curve $\mathcal{A}_1$ or zero dimensional. This confirms that  the Raynaud extension is semi-abelian or a two dimensional torus corresponding to the two boundary components respectively. 

\subsubsection{} Consider the Kuga-Satake abelian variety $\KS(X)$ that reduces  to the boundary of the (compactified) Siegal moduli space $\mathcal{A}_d$.
Write $Z_{\KS}$ as an extension of a torus by an abelian variety $B$. From the above discussion $B$ is parametrised by some Baily-Borel boundary component of $\mathcal{A}_d$. Hence, apriori $B$ is a $K$-point of $\mathcal{A}_{g}$ where $g<d$. However, we expect $B$ to have a nice description once we consider the two boundary components of the moduli space of K3 surfaces. The abelian variety $B$ is trivial if we assume that the K3 surface reduces to the zero dimensional boundary. While on the other hand, when the K3 surface reduces to a point on $\mathcal{A}_1$, then $B$ is a $K$-point in a sub-locus of the Baily-Borel boundary $\mathcal{A}_{d/2}$ of the Siegal moduli space. This is proved in Corollary \ref{prodec}. This sub-locus is precisely given by the map between the boundary components of the respective moduli spaces: $$\mathcal{A}_1\hookrightarrow \mathcal{A}_{d/2}: E\mapsto E^{d/2}$$ Hence, the abelian quotient $B$ of the Raynaud extension of $\KS(X)$ is expected to be isogenous to $d/2$ copies of an elliptic curve. The proof of this fact is presented in Section \ref{raynaudks} and exploits Hodge structures on the toroidal boundary that linearises the data of Raynaud extensions along with the Kuga-Satake map in characteristic zero. Using the theory of good integral models for toroidal compactifications, we deduce the results in characteristic $p$.

\subsubsection{Finiteness of the reduction of Hecke orbit.} A direct consequence of the monodromy results is the finiteness of the reduction of $p$-power Hecke orbit of an ordinary point of an orthogonal Shimura variety (and its associated Kuga-Satake abelian variety) with type II supersingular reduction. Following the same ideas of \cite{KLSS} Corollary 2.10, we prove this in Corollary \ref{finitehecke}. 
\subsubsection{Organisation of the paper}
This paper is organised as follows. In Section \ref{section3} and \ref{badred} we review Katz's proof for understanding monodromy at a super-singular point in $\mathcal{A}_1$, along with  the Tate uniformisation to compute the monodromy for abelian varieties with complete bad reduction. In Section \ref{raynauduniformisation} and \ref{6} we review Raynaud's uniformisation theorem for abelian varieties with semistable reduction and then prove the monodromy theorem for abelian surfaces with bad reduction. Moving on to K3 surfaces we  review briefly GSpin-Shimura data, Kuga-Satake construction and toroidal compactification in Section \ref{7} and \ref{8}. In Section \ref{raynaudks}, we compute the mixed Hodge structures on the boundary that determine the Raynaud extension of the Kuga-Satake abelian variety. Finally in Section \ref{10} and \ref{k3}, we prove the monodromy theorem for the Kuga-Satake abelian variety and orthogonal Shimura varieties.  

\subsubsection*{Acknowledgments}

I am grateful to my advisor, Ananth Shankar for proposing possible generalizations of Igusa's result and many valuable conversations about this paper. I would also like to thank Keerthi Madapusi Pera for clarifying comments about Raynaud extensions. 
\section{Large monodromy of elliptic curves in characteristic $p$}\label{section3}

 Throughout this section we fix $E/K$ to be an ordinary elliptic curve with supersingular reduction. We have the following theorem due to Igusa.
\begin{thm}[Igusa]\label{ss}
 Consider a sequence of points of $E(\overline{K}):$
$$y^{(1)}, y^{(2)}, y^{(3)} \dots$$
such that $p y^{(1)} = e$ and $py^{(n+1)} = y^{(n)}$
for $n\geq 1.$ 
Let $K(y^{(n)})$ denote the field extension obtained by adjoining $y^{(n)}$ to $K$. Then there exists some $n_0$ such that for all $n> n_0,$ the field extension $K(y^{(n)})^{\text{sep}}$ is a totally ramified extension of $K(y^{(n_0)})$.
\end{thm}
 Let $\overline{R}$ be valuation ring of $\overline{K}$ and $\mathscr{M}$ its maximal ideal. Let $\mathcal{E}$ over $R$ be the N\'eron model of $E/K$ with special fibre $E_0$. Denote by $\widehat{\mathcal{E}}$ the formal group of $E$.  Recall the following short exact sequence that we get from the reduction map: 
$$0\rightarrow  \widehat{\mathcal{E}}(\mathscr{M})\rightarrow \mathcal{E}(\overline{R})\rightarrow E_0(\overline{\F_q})\rightarrow 0$$ 
 We note that the $p$-power torsion is in the kernel of the reduction map above. Hence they lie in the $\mathscr{M}$-points of $\widehat{\mathcal{E}}$. Thus, in order to study the $p$-power torsion of $E$ we look at its formal group. This gives motivation for Katz's argument of the above theorem that we now describe below.
\begin{proof}
Consider the formal group $\widehat{\mathcal{E}}_R$, the formal group of $E/K$. 
For a suitable parameter $x,$ we consider the multiplication by $p$ map $[p]_{R}$ which is a power series of the form $g(t^p)$ for some $g(t)\in R[[t]].$ By the Weierstrass preparation theorem any power series can be decomposed as a product $(u\cdot h)(t)$ where $u$ is a unit  in $R[[t]]$  and $h(t)$ is a polynomial.  
The $p$-torsion of $\mathcal{E}(\overline{R})$ are the roots of $g$, hence it is enough to consider $h(t).$ We write: $$[p]_{R} = A_1(t)x^p+A_2(t)x^{p^2}$$
The formal group of the special fibre is supersingular and thus has height $2.$ Consequently,  $$v_K (A_1(t))\geq 1 \text{ and } v_K(A_2(t))=0$$ We first suppose that $v_K(A_1(t))=1$. In fact, we can choose such a representation of $[p]_R$, when $E$ is the universal elliptic curve around a supersingular point of $\mathcal{A}_1$.  Since the Frobenius acts by $x\mapsto x^p$, the iterates of $V$ act as $V^{(p^n)}(x) = A_1(t)^{p^n}x +A_2(t)^{p^n}x^p.$ We note that $[p](y^{(1)}) = 0,$ while $[p](y^{(n+1)}) = y^{(n-1)}$ for $i\geq 2$. This gives us a system of equations:

$$A_1(t)y^{(1)} +A_2(t)(y^{(1)})^p = 0$$
$$A_1^p(t)y^{(2)} +A_2(t)^p(y^{(2)})^p= y^{(1)}$$
$$\vdots$$
$$A_1^{p^{n+1}}(t)y^{(n+1)} +A_2(t)^{p^{n+1}}(y^{(n+1)})^p = y^{(n)}$$
Notice that $[K(y^{(n)})^{\insep}:K] = p^n$. In order to understand the separable part of the extension, we plot the Newton polygon to conclude the result iteratively. For instance, the Newton polygon from the first equation helps us conclude that $$v_K(y^{(1)}) = \frac{1}{p-1}=\frac{v_{K^{\sep}}(y^{(1)})}{e}$$
       We know that $e\leq [K(y^{(1)})^{\sep}:K] \leq p-1$. But $v_{K^{\sep}}(y^{(1)})\geq 1$ since $y^{(1)}$ reduces to zero modulo the maximal ideal.  This proves equality everywhere and the fact that $y^{(1)}$ is a uniformizer of $K(y^{(1)}).$ We carry on this process to get that $K(y^{(1)},y^{(2)}, \dots y^{(n)})^{\sep}/K$ is a totally ramified extension of degree $(p-1)p^n$.

Next, consider  $v_K(A_1(t)) = m > 1$. Then we see that  $v_{K^{\sep}}(y^{(1)})\leq m$ and we only get full ramification if this is an equality. However, as we draw subsequent Newton polygons, we get that $v(A_1(t)^{p^n}) = p^nv(A_1(t))$ which becomes large as $n$ increases. Thus, this point doesn't affect the Newton polygon for $n$ large enough. As a result,  the slope becomes $1/p$ for some sufficiently large $n_0$, and we see that $K(y^{(n)})/K(y^{(n_0)})$ is a totally ramified extension for all $n>n_0.$ 
\end{proof}

\begin{cor}
Consider the associated Galois representation associated to the $p$-power torsion of $E$ with supersingular reduction:
$$\rho_{\sups}:\Gal(K^{\text{sep}}/K)\rightarrow \Z_p^{\times}$$ Then the image of $\rho_{\sups}$ is open in $\Z_p^{\times}$. Moreover the inertia subgroup has finite index in the image.
\end{cor}
\begin{proof}
It follows from Theorem \ref{ss} that the image is infinite. As $\Z_p^{\times}$ is Hausdorff, the image of the compact Galois group is closed. Since the infinite closed subgroups of $\Z_p$ (and $\Z_p^{\times})$ are open, the first statement follows. The index of the image of the inertia subgroup equals the degree of the unramified extension obtained by attaching the $p$-power torsion to $K$. This is finite by Theorem $\ref{ss}$.
\end{proof}
\section{Abelian varieties with  totally bad reduction}\label{badred}
In this section we study the $p$-power torsion of abelian varieties that completely degenerate. The following explicit result helps us understand the Galois action. 

\begin{prop}(Tate uniformisation)
Let $A$ be a $g$-dimensional abelian variety over $K$ with totally bad reduction. Up to a base change, we may assume $A_0\simeq \bbG^g_{m,\F_q}$. Then there exists multiplicatively independent elements $q_1,\dots q_g$ in $\bbG_m^g(\overline{K})$ with positive valuation such that we have a Galois invariant group isomorphism of $\overline{K}$-points:
$$\alpha: A(\overline{K})\rightarrow \bbG^g_m(\overline{K})/\langle{q_1,\dots ,q_g}\rangle$$
where $\langle{q_1,\dots q_g}\rangle$ is the multiplicative lattice generated by the $g$-elements in $\bbG^g_m(\overline{K}).$
\end{prop}

\begin{thm}
We keep the setting of the above theorem. Let $y^{(1)}, y^{(2)},y^{(3)}\dots$ be a sequence of $p$-power torsion points of $A(\overline{K})$ such that $py^{(1) } = e$ and $py^{(n+1)} = y^{(n)}$ for all $n\geq 1.$ Denote by $K(y^{(n)})$ the field extension obtained from attaching $p^{n}$-torsion point. Then $K(y^{(n)})^{\text{sep}} = K$. Moreover, $K(y^{(n)})^{\text{insep}} = K(q_1^{1/{p^n}}, \dots, q_g^{1/{p^n}})$.
\end{thm}
\begin{proof}
Since $\alpha$ is a group isomorphism and Galois invariant, it is enough to look at the images of $p$-power torsion points under the map $\alpha.$ The torsion in $\bbG^g_m(\overline{K})/\langle{q_1,\dots q_g}\rangle$ is given by $p$-power roots of the generators of the lattice which give us inseparable extensions over $K$.
\end{proof}
 
\begin{cor}
Let $A$ be an ordinary abelian variety over $K$ with completely bad reduction. Then the image of the associated Galois representation is trivial.
\end{cor}
\section{Raynaud's uniformization for degenerating abelian varieties}\label{raynauduniformisation}
 In higher dimensions, Raynaud proved a uniformization theorem for abelian varieties with semi-stable reduction that generalizes Tate's theorem. In this section we review Raynaud's results and use them to compute the monodromy representation in the next section. Our main references are \cite{raynaud} and \cite{raynaudgen}.

 Let $A$ be an abelian surface over $K$ with semi-stable reduction over $K.$ Then its reduction $A_0$ over $k$ is an extension of a torus $T_0$ by an abelian variety $B_0$: 
\begin{equation}\label{eq1}
0\rightarrow T_0\rightarrow A_0\rightarrow B_0\rightarrow 0
\end{equation}
 
 Raynaud's results in \cite{raynaud} construct a $p$-adic uniformizing space as follows: we complete the identity component of the N\'eron model of $A$ along the special fibre to get a formal group which we denote by $\widehat{A}.$  The torus $T_0$ also lifts to a formal torus $\widehat{T}$ so that $\widehat{A}$ is an extension of a formal torus by a formal abelian scheme over $R$. We define $\widehat {B}$ to be the quotient. We therefore get the following exact sequence:
\begin{equation}\label{eq2}
0\rightarrow \widehat{T}\rightarrow  \widehat{A}\rightarrow \widehat{B}\rightarrow 0.
\end{equation}

The Raynaud generic fibre of these formal groups have a structure of a rigid analytic space, giving us an exact sequence of rigid analytic spaces:
 
\begin{equation}\label{eq3}
 0\rightarrow \widehat{T}_{\rig}\rightarrow \widehat{A}_{\rig}\rightarrow \widehat{B}_{\rig} \rightarrow 0
 \end{equation}
 
 We note that $\widehat{A}_{\rig}\subset A^{\an}$ is an open analytic subgroup of the analytification of $A$.  It is equal to $A^{\an}$ if and only if $A$ has good reduction. In that case $T=0.$ By construction, $\widehat{T}_{\rig}$ is the ``group of units" inside a full analytical torus. We can further extend these maps to get an extension by a full analytical affine  torus that is $ T_{\rig}\simeq (\mathbb{G}_m^{\an})^d$ (see \cite{raynaudgen}, Section 1). We therefore get a rigid analytic extension $Z$ which will be the parametrizing space for $A^{\an}:$  

\begin{equation}\label{eq4}
0\rightarrow T_{\rig}\rightarrow Z\rightarrow B_{\rig} \rightarrow 0 
\end{equation}
The parametrisation is recorded in the following theorem. See \cite{raynaudgen} for further details. We mainly write the statements that will be useful for us. 

\begin{thm}(See Theorem 1.2 in \cite{raynaudgen})
 Keeping the notation above, we have the following:
\begin{enumerate}
    \item The closed immersion $\widehat{T}_{\rig}\hookrightarrow \widehat{A}_{\rig}$ extends uniquely to a rigid analytic group map $\hat p:T_{\rig}\rightarrow A^{\an}$
    \item The open immersion $\widehat{A}_{\rig}\hookrightarrow A^{\an}$ extends uniquely to a surjective rigid analytic group morphism $p: Z\rightarrow A^{\an}$
    \item The kernel $M$ of the map $p$, is a lattice in $Z$
 whose rank equals the dimension of the torus $T_{\rig}$. The rigid analytic morphism $Z/M\rightarrow A$ that we obtain from $p$ is an isomorphism.
   \end{enumerate}
\end{thm}
The exact sequence in \ref{eq4} is algebraisable (see \cite{raynaudgen}, Section 1). We will  drop the subscript $``\rig"$ in \ref{eq4} to refer to their algebraic counterpart. 
\section{Monodromy theorem for abelian surfaces with semi-abelian reduction}\label{6}
From this section on wards, we work in dimension two and fix $A$ to be a simple abelian surface over $K$ with semi-stable reduction. Indeed, it is enough to assume that $A$ is simple. For if $A$ is isogenous to a product of elliptic curves, then at least one of them has bad reduction, reducing us to the case of dimension one. Further, we denote by $Z$ its Raynaud extension. Then the abelian variety $B$ in the exact sequence \ref{eq4} is an elliptic curve and $T$ is one dimensional.  We first record some results about the structure of $p$-power torsion of $A$ using Raynaud's uniformisation  result. 
\subsection{Galois action on the $p$-power torsion of $A$} Let $y^{(n)}$ and $z^{(n)}$ denote the bases of $p^n$-torsion of $A(\overline{K})$. Note that for all $n\geq 1$, $y^{(n)}$  generates the $p^n$-torsion of $Z$ and $z^{(n)}$ is such that $p^n z^{(n)} = \lambda$ where $\lambda$ is the generator of $M$.

\begin{lem}
Let $\varphi: Z\rightarrow B$ denote the extension map between the parametrizing space and the elliptic curve. Then $B$ is an ordinary elliptic curve and $\varphi$ maps $y^{(n)}$ and $z^{(n)}$ injectively onto $B(\overline{K})$ for all $n\geq 0.$ Moreover, $\varphi(y^{(n)})$ maps onto $B[p^n](\overline{K})$, while $p^n\varphi(z^{(n)}) =\varphi(\lambda)$.
\end{lem}
\begin{proof}
The elliptic curve $B$ has to be ordinary, otherwise the $p$-power torsion of $Z$ lies in the kernel of $\varphi.$ However, since $T$ does not have any $p$-power torsion over $\overline{K}$, this cannot happen. The map is an injection on $y^{(n)}\in Z[p^n](\overline{K})$ because of the same reason. While if  $\varphi(z^{(n)}) = e$, then $p\varphi(z^{(n)}) = \varphi(\lambda) = e$. But then $\langle{\lambda}\rangle \subset T$ implies that $T/M\hookrightarrow A$, contradicting that $A$ is simple. Finally, the map $\varphi$ is a group homomorphism and hence the assertion about the images of $y^{(n)}$ and $z^{(n)}$ is true for all $n\geq 1.$
\end{proof}
\begin{lem}
Let $\sigma\in \Gal(K(y^{(n)}, z^{(n)})/K).$ Then $\sigma$ preserves the $p$-power torsion $Z[p^n](\overline{K})$, while the action on $z^{(n)}$ for all $n\geq 0$ is such that $\sigma(z^{(n)}) = z^{(n)}+ ky^{(n)}$ for some $1\leq k\leq p-1$.
\end{lem}
\begin{proof}
The Galois group preserves the $p$-power torsion of $Z[p^n](\overline{K}).$ While note that $p^n\cdot(\sigma(z^{(n)}) - z^{(n)}) = \lambda -\lambda = 0.$ Hence $\sigma(z^{(n)}) - z^{(n)}\in Z[p^n](\overline{K}).$ 
\end{proof}

 Before we prove  Theorem \ref{thm1}, we summarize Raynaud's results and the above lemmas. We get the following diagram of rigid analytic spaces such that:

$$\xymatrix{ &M\ar[d]&\\\bbG_m \ar[r] & Z \ar[r] \ar[d]&B& \\
& A&  }$$

\begin{enumerate}
\item $B$ is an ordinary elliptic curve.
\item $M$ is a lattice in $Z$ of rank 1
\item It is enough to understand the Galois action of the images of the  points in $Z$ under the map $Z\rightarrow B$ that parametrise $p$-power torsion of $A$.
\item The action of an element $\sigma$ of the Galois group $K^{\sep}/K$ on the bases elements $y^{(n)}$ and $z^{(n)}$ is represented by the following matrix in $\GL_2(\Z/p^n\Z)$:
$$\sigma = \begin{pmatrix}
*&*\\
0&1

\end{pmatrix}
$$

\end{enumerate}
\subsection{Ramification.} Now, since it is enough to understand the images in $B$, the Galois action on the $p$-power torsion points depends on the reduction type of $B$. That is, it depends on whether the abelian quotient $B_0$ in \ref{eq1} is ordinary or supersingular.

\subsubsection{The case when $B_0$ is ordinary.} By abuse of notation we write the images of the bases of $p$-power we torsion in $B$ as $y^{(n)}$ and $z^{(n)}$ as well. Since $B$ has ordinary reduction, the field extension generated by $y^{(n)}$ is unramified as the reduction map $\mathcal{B}(\overline{R})\rightarrow B_0(\overline{\F_q})$ is injective on the $p^n$-torsion of $\mathcal{B}$ where $\mathcal{B}$ is the N\'eron model of $B$. Moreover, it follows from Lemma \ref{totram} below that once we attach the $z^{(n)}$ to $K(y^{(n)})$, the field extension $K(z^{(n)}, y^{(n)})^{\sep}$ over  $K(y^{(n)})$ must be totally ramified.  This proves that the inertia subgroup fixes the bases element $y^{(n)}$ for $n\geq 1$ and hence its action is unipotent.

\subsubsection{The case when $B_0$ is supersingular.} In this case, we get a result  analogous to Igusa's theorem.
\begin{prop}\label{semiabelian}
Let $A$ be a simple abelian surface over $K$ with semi-abelian reduction such that the abelian variety $B_0$ in exact sequence \ref{eq1} is a super-singular elliptic curve. Let $Z$ be its uniformizing space and $M = \langle{\lambda\rangle}$ be the lattice in $M$ such that $A\simeq Z/M$.  Consider a sequence of points in $G(\overline{K})$: $$y^{(1)},y^{(1)},y^{(2)}, \dots \text{ and } z^{(1)}, z^{(1)}, z^{(2)}\dots$$
such that $py^{(1)}=e$ and $py^{(n+1)}=y^{(n)}$, $pz^{(1)}=\lambda$ and $pz^{(n+1)}=z^{(n)}$ for $n\geq 1.$ Then there exists some $n_0$ such that for all $n>n_0$, the field extension $K(y^{(n)}, z^{(n)})/K(y^{(n_0)}, z^{(n_0)})$ is a totally ramified extension.
\end{prop}

\begin{proof}
As in the previous case, we first consider (the image of) $y^{(n)}$ in $B$ for $n\geq 1$. Since $B$ has super-singular reduction, by Theorem \ref{ss} we know that there exists some $n_0$ such that  $K(y^{(n)})/K(y^{(n_0)})$ is totally ramified for all $n>n_0.$ Next, consider (the image of) $z^{(n)}$ for $n\geq 1$. We show that $K(z^{(n)},y^{(n)})/K(y^{(n)})$ is a totally ramified for all $n >n_0$ as a separate lemma below.
\end{proof}
\begin{lem}\label{totram}
Let $B$ be an ordinary elliptic curve over $K$ with supersingular reduction. Let $\lambda$ be any $K$-point of $B.$ Take a sequence of points in $B(\overline{K})$ $$z^{(1)},z^{(2)},z^{(3)}\dots$$ such that $pz^{(1)} =\lambda$ and $p z^{(n)} = z^{(n+1)}$. Let $y^{(n)}$ be the generator of the $p^n$-power torsion points of $B$. Then there exists some $n_0$ such that field extension $K(z^{(n)},y^{(n)})/K(y^{(n)})$ is totally ramified for all $n> n_0.$
\end{lem}
\begin{proof}
We first note that if $z^{(1)}$ is in $K$ then we can apply the transformation $z^{(n)}\mapsto z^{(n)} - z^{(1)}$ for $n\geq 1$ and that doesn't change the Galois action as we're translating by a $K$-point. Hence $z^{(n)}$ lie in $K(y^{(n)})$ for all $n\geq 1$. Therefore, we deal with the non-trivial case when $z^{(1)}$ doesn't lie in $K$. We know that $\Gal(K(z^{(n)},y^{(n)})/K(y^{(n)})) = \Z/p^n\Z$.  Hence all the sub-fields of $K(z^{(n)},y^{(n)})$ are of the form $K(z^{(k)}, y^{(n)})$ with $k\leq n$. Suppose there is an unramified sub-field of $K(z^{(n)},y^{(n)})$ then it contains $K(y^{(n)},z^{(1)})$. But this is a contradiction as we know that Galois group of $K(z^{(1)},y^{(n)})/K$ is not abelian. 
\end{proof}
This concludes the proof of all the cases in Theorem \ref{thm1}. The rest of the paper is devoted to proving an analogous result for K3 surfaces. To that end we start by reviewing  GSpin Shimura varieties and the Kuga-Satake construction.
\section{GSpin Shimura varieties and the Kuga-Satake construction} \label{7}

\subsection{Siegal Shimura varieties} We fix  notation as follows. 
Let $(H,\psi)$ be a $2g$ dimensional symplectic space over $\Q$ and $\GSp(H,\psi)$ be the group of symplectic similitudes. We have the Siegal Shimura data $(\GSp_{2g},X)$ where $X$ is the associated Hermitian space such that $\GSp$ acts on the space by conjugation. For a compact open $K^{\dagger}\subset \GSp(\A_f)$, we denote by $\Sh_{K^{\dagger}}(\GSp, X)$ to be the associated Siegal Shimura variety defined over $\Q$. When $K^{\dagger}$ is hyperspecial\footnote{A compact open subgroup $K^{\dagger}\subset G(\mathbb{A}_f)$ of a reductive group $G$ is called  hyperspecial at $p$ if $K^{\dagger}_p = G_{\Z_p}(\Z_p)$ where $G_{\Z_p}$ is a model over $\Z_p$ with generic fibre $G$.} at $p$, by \cite{intmod},\cite{2adicintmod}, it admits a smooth integral canonical model $\mathcal{A}_{g,K^{\dagger}}$ over $\Z_{(p)}.$ The Siegal Shimura variety is the moduli space of (polarised) $g$-dimensional abelian varieties with given a level structure. Hence it carries a universal family of abelian scheme which we denote as $\mathscr{A}^{\univ}\rightarrow \Sh_{K^{\dagger}}(\GSp, X)$.
The universal abelian scheme extends to the canonical integral model and  we denote by $\mathscr{A}_{\F_p}^{\univ}\rightarrow\mathcal{A}_{g,\F_p}$ the mod $p$ Siegal Shimura variety along with the universal family in characteristic $p$.
\subsection{GSpin Shimura data} Let $(L,q)$ be a quadratic $\Z$-lattice of signature $(n,2)$ that is self dual at $p$. Let $V = L\otimes_{\Z} \mathbb{Q}$ be the vector space of dimension $n+2$ with the associated  bilinear form given by $q$ such that  such that $q|_{L}$ belongs to $\Z.$  Let $Cl(-)$ denote the Clifford algebra. The Clifford algebra comes with a $\Z/2\Z$ grading $Cl^{+}(-)\oplus Cl^{-}(-)$. Let $G = \GSpin(L, q)$ be the group of spinor similitudes of $L$. That is, for any $\Q$-algebra $S$, we have:
$$\GSpin(L,q)(S) = \{x\in (Cl^{+}_S)^{\times} \mid xL_Sx^{-1} = L_S\}$$
Via the   map $ \GSpin(V, q)\rightarrow \SO(V,q); g\mapsto (v\mapsto g\cdot v\cdot g^{-1})$,
the group $G(\mathbb{R})$ acts on the Hermitian domain   
$\mathcal{D} = \{z\in V_{\C}\mid (z,z) =0, (z,\overline{z})<0\}/\C^{\times}$.
 The pair $(G,\mathcal{D})$ is the GSpin Shimura datum with reflex field $\Q.$ Let $K^{\sharp}\subset G(\A_f)$ be a compact open subgroup contained in $G(\A_f)\cap Cl(L\otimes \widehat\Z)^{\times}$ such that $K^{\sharp}$ is hyperspecial at $p$.
We denote by $\Sh_{K^{\sharp}}(G,\mathcal{D})$, the associated Shimura variety over $\Q$. We have that $$ \Sh_{K^{\sharp}}(G,\mathcal{D})(\C) = G(\Q)\backslash\mathcal{D}\times G(\A_f)/K^{\sharp}$$
We note that orthogonal Shimura varieties are finite \'etale quotients of GSpin Shimura varieties and hence we will deduce monodromy results in this case using this fact. 
By work of \cite{intmod} and \cite{spinsh} these admit a smooth canonical integral model over $\Z_p$ and hence we define their mod $p$ reduction to be GSpin and orthogonal Shimura varieties over $\F_p$.
We denote by $\mathcal{M}_{K^{\sharp}}$ the integral model for $\Sh_{K^{\sharp}}(G_{\Q},\mathcal{D})$ over $\Z_{(p)}$ and $\mathcal{M}_{\F_p}$ its mod $p$ reduction.   
\subsection{The Kuga-Satake construction}\label{KS}  An extremely useful fact about GSpin Shimura varieties is that they are of Hodge type. That is, there is an embedding of Shimura data $(G,\mathcal{D})\hookrightarrow (\GSp, X)$ over $\Q$. This gives us the \textit{Kuga-Satake embedding} on the level of Shimura varieties over $\Q$:
$$\iota :\Sh_{K^{\sharp}}(G,\mathcal{D})\hookrightarrow \Sh_{K^{\dagger}}(\GSp, X) $$
The pullback of the universal abelian scheme over the Siegal Shimura variety yields the Kuga-Satake abelian scheme $\mathscr{A^{\KS}}\rightarrow \Sh_{K^{\sharp}}(G,\mathcal{D})$.  We review briefly the Kuga-Satake map for a (polarised) K3 crystal of a point $[z]\in \mathcal{D}$. Let $\mathbb{S}$ denote $\Res_{\C/\R}\bbG_m$. We know that the associated unique homomorphism $h_{[z]}: \mathbb{S}\rightarrow G_{\R}$ induces a weight two Hodge structure on $V_{\C}$ with a positive definite two dimensional space $V^{2,0}\oplus V^{0,2}.$  The reductive group $\GSpin(V,q)$ acts on $C = Cl(V,q)$ by multiplication on the left. We call this resulting representation by $H$. Note that right multiplication by $C$ gives it a $C$-module structure and a $\GSpin(V,q)$ stable $\Z/2\Z$ grading $H^{+}\oplus H^{-}.$ The subset $U(H)\subset \End_C(H)$ preserving the grading on $H$ are precisely $(C^{+})^{\times}.$ We know that $\GSpin(V,q)\subset (C^{+})^{\times}$.  Once we pick a generator $e_1+ie_2$ of $V^{0,2}$, we get a bases of two orthogonal vectors $e_1$ and $e_2$ such that $q(e_1) = q(e_2) = 1$ spanning $V^{2,0}\oplus V^{0,2}\cap V_{\R}$.  The map $J: Cl^{+}(V_{\R})\rightarrow Cl^{+}(V_{\R}); v\mapsto e_1\cdot e_2 \cdot v$
 induces a complex structure (and hence a weight one Hodge structure) on $Cl(V_{\R})$ since $J^{2}\equiv -1.$ In terms of representation of the Deligne torus, the Hodge structure on $V$ corresponds to a morphism 
$\mathbb{S}\rightarrow \SO(V,q)$. By the above discussion we get a lift of this morphism to $\mathbb S\rightarrow \GSpin(V,q)\rightarrow GL(Cl(V))$ and this corresponds to the weight one Hodge structure. Moreover, given two vectors $f_1$ and $f_2$ such that $q(f_1) = q(f_2)>0$, we define a polarization on $Cl^{+}(V_{\R})$ using the trace map
$$\psi: Cl^{+}(V)\times Cl^{+}(V)\rightarrow \Q; (v_1,v_2) = tr(f_1\cdot f_2\cdot v_1^{*}, v_2) $$

The Kuga-Satake abelian variety over $\C$ is the abelian variety we get from the complex torus  $Cl^+(V_{\R})/Cl(V)$ along with the polarisation defined above. Moreover, these descend to the field of definition of the associated K3 surface in characteristic $0$. 

On the level of moduli spaces, this corresponds to lifting a point of the orthogonal Shimura variety to the GSpin Shimura variety which further embeds inside the Siegal moduli space. The classical construction extends to canonical integral models as well (see for instance \cite{spinsh}). We denote by $\mathscr{A}^{\KS}_{\F_p}\rightarrow \mathcal{M}_{\F_p}$ the mod $p$ family of Kuga-Satake abelian scheme over GSpin Shimura variety over $\F_p$. Furthermore, the Kuga-Satake construction respects the ordinary stratum. 
\begin{lem}\label{ordks}
Let $X$ be an ordinary K3 surface over $K$.  Then the associated Kuga-Satake abelian variety $\KS(X)$ over $K$ is ordinary as well. 
\end{lem}
\begin{proof}
This is Theorem 7.8 in \cite{ogus} for perfect fields. Since the property of being ordinary stays the same after base change, we argue in our setting by replacing $K$ with its perfection to conclude the result.  
\end{proof}
\section{Toroidal compactifications and mixed Hodge structures}\label{8}

\subsection{Toroidal compactifications over $\mathbb{C}$}
This section formalizes the discussion in Subsection \ref{philosophy}. We mainly follow the exposition in \cite{keerthitor} and \cite{salimk3} to describe toroidal compactifications over $\C$ and the structure of the boundary components. We denote by $G$ to be the algebraic group $\GSpin$ or $\GSp$. Let $\Sh_{K^{\dagger}}(G,\mathcal{D})$ be the associated Shimura variety where  $\mathcal{D}$ be the Hermitian domain and $K^{\dagger}\subset G(\mathbb{A}_f)$ is the compact open which hyperspecial at $p$. Let $\mathcal{S}_{K^{\dagger}}$ denote the canonical integral model of $\Sh_{K^{\dagger}}(G,\mathcal{D})$ defined over $\Z_{(p)}.$ By work of \cite{pink}, there exists a proper toroidal compactification $$\Sh_{K^{\dagger}}(G,\mathcal{D})\hookrightarrow \Sh_{K^\dagger}(G,\mathcal{D})^{\text{tor}}$$
in the category of Deligne-Mumford stacks over $\Q$. By \cite{keerthitor}, $\mathcal{S}_{K^{\dagger}}$ admits a toroidal compactification $\mathcal{S}^{\text{tor}}_{K^{\dagger}}$ that extends the compactification of $\Sh_{K^{\dagger}}(G,\mathcal{D})$ over $\Q$. The compactification $\mathcal{S}^{\text{tor}}_{K^{\dagger}}$ depends on a certain cone decomposition (see \cite{keerthitor}, Section $2$ and $4$) and consists of a stratification by locally closed subschemes that can be described by the following components. 
\begin{defn}
Let $P\subset G$ be a parabolic subgroup. We say $P$ is admissible if $P$ is either maximal in $G$ or equals $G.$
\end{defn}
\begin{defn} A cusp label representative is a triple $\Phi = (P,\mathcal{D}^{\circ}, h)$ where $P$ is an admissible parabolic subgroup of $G$, $\mathcal{D}^{\circ}\subset \mathcal{D}$ is a connected component and $h\in G(\mathbb{A}_f)$. 

\end{defn}

 To any cusp label representative, we can attach the data of a rational boundary component as follows. We denote by $U_{\Phi}$ the unipotent radical of $P$ and let $W_{\Phi}$ denote its centre. Let $Q_{\Phi}$ be the unique normal subgroup of $P$ as defined in 
\cite{pink} Section 4.7. Let $\mathcal{D}_{\Phi} = Q_{\Phi}(\R)W_{\Phi}(\C)\mathcal{D}^{\circ}$ and  $K_{\Phi} = hK^{\dagger}h^{-1}\cap Q_{\Phi}(\A_f)$. The pair $(Q_{\Phi},\mathcal{D}_{\Phi})$ defines a mixed Shimura variety whose $\C$-points are given by $$\Sh_{K_{\Phi}}(Q_{\Phi},\mathcal{D}_{\Phi})(\C) = Q_{\Phi}(\Q)\backslash\D_{\Phi}\times Q_{\Phi}(\A_f)/K_{\Phi} $$
Further, let $\overline{Q}_{\Phi} = Q_{\Phi}/W_{\Phi}$ and let $\overline{\mathcal{D}}_{\Phi} = W_{\Phi}/\mathcal{D}_{\Phi}$ and $\overline{K}_{\Phi}\subset \overline{Q}_{\Phi}$ be the image of $K_{\Phi}$. Then $(\overline{Q}_{\Phi},\overline{\mathcal{D}}_{\Phi})$ gives a mixed Shimura data such that $\C$-points of the mixed Shimura variety given by 
$$\Sh_{K_{\Phi}}(\overline{Q}_{\Phi},\overline{\mathcal{D}}_{\Phi})(\C) = \overline{Q}_{\Phi}(\Q)\backslash\overline{\D}_{\Phi}\times \overline{Q}_{\Phi}(\A_f)/\overline{K}_{\Phi}$$
 Next, let $V_{\Phi} = Q_{\Phi}/W_{\Phi}$, $G_{\Phi,h} = Q_{\Phi}/U_{\Phi}$, $\mathcal{D}_{\Phi,h} = V_{\Phi}/\mathcal{D}_{\Phi}$ and $K_{\Phi,h}\subset G_{\Phi, h}$. Then we can attach a Shimura variety to the pair $(G_{\Phi,h}, \D_{\Phi, h})$ with $\C$-points parametrised by
$$\Sh_{K_{\Phi}, h}(G_{\Phi,h}, \mathcal{D}_{\Phi,h})(\C) = 
G_{\Phi,h}(\Q)\backslash\D_{\Phi,h}\times G_{\Phi,h}(\A_f)/K_{\Phi,h}$$
From this data, we get a tower of the mixed Shimura varieties: 
$$\Sh_{K_{\Phi}}(Q_{\Phi}, \mathcal{D}_{\Phi}) (\C)\rightarrow \Sh_{\overline{K}_{\Phi}}(\overline{Q}_{\Phi}, \overline{\mathcal{D}}_{\Phi}) (\C)\rightarrow \Sh_{K_{\Phi}, h}(G_{\Phi,h}, \mathcal{D}_{\Phi,h})(\C)$$

 By \cite{keerthitor}, these admit canonical integral models which we denote by $\mathcal{S}_{K_{\Phi}}(Q_{\Phi}, \mathcal{D}_{\Phi}),\mathcal{S}_{\overline{K}_{\Phi}}(\overline{Q}_{\Phi}, \overline{\mathcal{D}}_{\Phi})$ and $\mathcal{S}_{K_{\Phi}, h}(G_{\Phi,h}, \mathcal{D}_{\Phi,h})$ over $\Z_{(p)}$ extending the canonical models over the reflex field $\Q.$ 
Once we choose a cone decomposition, the component $\mathcal{S}_{K_{\Phi}}(Q_{\Phi}, \mathcal{D}_{\Phi})$ embeds inside a mixed Shimura variety over $\mathcal{S}_{\overline{K}_{\Phi}}(\overline{Q}_{\Phi}, \overline{\mathcal{D}}_{\Phi})$  whose completion along the image of this component is isomorphic to the completion of $\mathcal{S}^{\text{tor}}_{K^{\dagger}}$ along the respective stratum (see \cite{keerthitor}, Theorem 4.1.5).
\begin{rem}
The mixed Shimura variety $\mathcal{S}_{\overline{K}_{\Phi}}(\overline{Q}_{\Phi}, \overline{\mathcal{D}}_{\Phi})$ is the toroidal boundary. The boundary component $\mathcal{S}_{K_{\Phi}, h}(G_{\Phi,h}, \mathcal{D}_{\Phi,h})$ is the Baily-Borel boundary which is an honest Shimura variety. For example, consider $\mathcal{A}_g$, then the Baily-Borel boundary of its compactification is either zero dimensional or a union of lower dimensional Siegal Shimura varieties: $\mathcal{A}_{1}\cup \mathcal{A}_{2}\cup\dots \cup \mathcal{A}_{2g - 1}$. Whereas for $\GSpin$ and orthogonal Shimura varieties, the Baily-Borel boundary is either zero dimensional or the modular curve $\mathcal{A}_1$ (see \cite{boundarygspin}, Section 1).  

\end{rem}

\subsubsection{Variation of mixed Hodge structures on the boundary}
For this section, we let $G=\GSp(H)$. The mixed Shimura variety $\Sh_{K_{\Phi}}(Q_{\Phi}, \mathcal{D}_{\Phi}) (\C)$ carries a variation of mixed Hodge structures of weights $(0,0),(-1,0), (0,-1)$ and $(-1,-1)$. The filtrations that determine the mixed Hodge structure are described as follows. Let $P_{\Phi}$ be an admissible parabolic subgroup, associated to a cusp label representative $\Phi$. For every $y\in \mathcal{D}_{\Phi}$, we get a co-character $\mu_{y}:\bbG_m\rightarrow \mathbb{S}_\C\rightarrow P_{\Phi,\C}$ where the map $\bbG_m\rightarrow \mathbb{S}_{\C
}$ is given by $z\mapsto (z,z)$. This splits the following ascending filtration fixed by $P_{\Phi,\C}$ on $H_{\C}$. 
$$W_{\bullet}V :W_{-3}H =0\subset W_{-2}H = I\subset W_{-1}H = I^{\perp}\subset W_{0}H =H.$$

 Here $I$ is the isotropic subspace stabilized by $P_{\Phi}$. In addition, we get an ascending filtration $F^{\bullet}V$
from the other co-character $\bbG_m\rightarrow \mathbb{S}_\C\rightarrow P_{\Phi,\C}$ where the first map is given by $z\mapsto (z,1)$. Then the pair $(W_{\bullet}V, F^{\bullet}V)$ determines   a (partly) polarised  mixed Hodge structure. The graded piece $\gr_1(H) = I^{\perp}/I$ corresponds to a polarised weight one Hodge structure while $\gr_0(H) =V/I^{\perp}$ is a weight zero Hodge structure. Locally this corresponds to the data of the Raynaud extension $(Z,\alpha)$  such that $\alpha: \Z^{r}\hookrightarrow Z$ is a lattice in $Z$. Moreover, the toroidal boundary $\Sh_{\overline{K}_{\Phi}}(\overline{Q}_{\Phi}, \overline{\mathcal{D}}_{\Phi}) (\C)$ admits a variation of mixed sub-Hodge structure of weights $(-1,0),(0,-1)$ and $(-1,-1).$ Via the first  map (1) of the tower
$$\Sh_{K_{\Phi}}(Q_{\Phi}, \mathcal{D}_{\Phi}) (\C)\xrightarrow{(1)} \Sh_{\overline{K}_{\Phi}}(\overline{Q}_{\Phi}, \overline{\mathcal{D}}_{\Phi}) (\C)\xrightarrow{(2)} \Sh_{K_{\Phi}, h}(G_{\Phi,h}, \mathcal{D}_{\Phi,h})(\C)$$
we lose information of the lattice. That is, locally the toroidal boundary parametrises the Raynaud extension $Z$. The Baily-Borel boundary carries a variation of polarised weight one pure Hodge structure on $\gr_1(H) = I^{\perp}/I$ with the ascending filtration: $$F^{0}I^{\perp}/I \subset I^{\perp}/I$$ 
Hence the second map $(2)$ corresponds to taking the quotient of the Raynaud extension to get its abelian quotient. 

\section{Raynaud extension of the Kuga-Satake abelian variety}\label{raynaudks}
 Let $(G,\mathcal{D})$ be the $\GSpin$-Shimura data and $\iota$ be the Kuga-Satake embedding $(G,\mathcal{D})\hookrightarrow (\GSp, X)$ data  over $\Q$. Since GSpin Shimura varieties are covers of orthogonal Shimura varieties, there is a one to one correspondence between their parabolic subgroups and hence, their boundary components. The boundary components are given by two kinds of admissible parabolic subgroups that correspond to either one dimensional  or  zero dimensional boundary components. Let  $P\subset G$ be an admissible parabolic subgroup, then there exists a unique minimal parabolic subgroup $P^{\prime}\subset \GSp$ containing $\iota(P)\subset P^{\prime}$ (see \cite{keerthitor}, Section 2.1.28). This parabolic subgroup determines the weight filtration of the mixed Hodge structure on the boundary of Siegal Shimura varieties and hence the Raynaud extension of the Kuga-Satake abelian variety. We compute this below. 
\subsubsection{Weight filtration on the one dimensional boundary} We follow \cite{spinsh}, Section 1.9 for the calculations in this section.  Let $(V,q)$ be the  quadratic space associated to $(G,\mathcal{D})$ and $H = Cl(V)$ be the Clifford algebra.  Let $G_0 = \SO(V,q)$.   Let $P_0\subset G_0$ be a parabolic subgroup and $P\subset G$ denote the corresponding parabolic subgroup of $G$. We know that this gives a weight filtration $W_{\bullet}V$ 

$$0=W_{-2}V\subset W_{-1}V = I\subset W_{0}V=I^{\perp} \subset W_{1}V = V$$

Let $\mu_0: \bbG_m\rightarrow \SO(V)\subset \GL(V)$ be the co-character that splits this filtration to give us a direct sum
$$I_{-1}\oplus I_0\oplus I_{1}$$ where $I_{-1} = I$ has a weight $-1$ Hodge structure and $I^{\perp} = I\oplus I_0$.  The co-character acts as multiplication by $z^{i}$ on $I_{i}$ for $-1\leq i\leq 1$. We find a lift $\mu: \bbG_m\rightarrow \GSpin(V,q)$ of the co-character $\mu_0$ that describes the associated weight filtration $W_{\bullet}H$ of $P.$  Since $I$ is isotropic and acts on $C$ by multiplication on the left, we have that $\wedge^{\bullet} I\hookrightarrow \End_{\Q}C$. We denote by $\im(\wedge^i I)$, the union of the images of endomorphisms $\wedge^i I$. Note also that $\im(\wedge^i I)\subset \im(\wedge^{i-1}I)$. For $i=0,1,\dots r+1,$ we define the weight filtration $W_{\bullet}H$ to be $W_{-i}H = \im(\wedge^{i}I_{-1})$ on $H$ where $r$ equals the number of $\Q$-generators of $I$.  Consider the ascending filtration $\im(\wedge^{r-i} I_{1})$ and define $H_{-i} = \im(\wedge^{r-i} I_{1})\cap I_{-i}.$ This defines a splitting of the filtration $W_{\bullet} H$. Now define the co-character $\mu: \bbG_m\rightarrow \GL(V)$ that acts as $z^{-i}$ on $H_{-i}$. Each $H_{-i}$ is $C$-stable and $\mu(\bbG_m)$ preserves the $\Z/2Z$ grading on $H_{-i}$ and hence on $H$. Thus, $\mu(\bbG_m)$ factors through $C^+$. We further check that the composition of $\mu$ with the conjugation map from $\GSpin(V,q)\rightarrow \SO(V,q)$ gives us $\mu_0$. This is because 

  \[v\cdot H_{-i} \subset \begin{cases} 
     H_{-i-1} & v\in I_{-1} \\
 H_{-i}    &  v\in I_0\\
     H_{-i+1} & v\in I_{1} 
   \end{cases}
\]

Writing $\mu(z)\cdot( v\cdot H_{-i}) = \mu(z)\cdot v\cdot \mu(z)^{-1}\mu(z)\cdot H_{-i}$ and using the above observation we deduce that

 \[\mu(z)\cdot v\cdot \mu(z)^{-1} =  \begin{cases} 
     z^{-1}\cdot v & v\in I_{-1} \\
v   &  v\in I_0\\
     z\cdot v & v\in I_{1} 
   \end{cases}
\]

Hence $\mu$ factors through $\GSpin(V,q)$ and lifts $\mu_0$.
As discussed in Section \ref{KS}, left multiplication by $\GSpin(V,q)$ on $H$, gives a co-character $\bbG_m\rightarrow \GSp(H)$. 
Therefore, the filtration on $W_{\bullet}H$ is the weight filtration on the boundary of the Shimura variety associated to $\GSp(H,\psi)$.

\subsubsection{Type II reduction} Suppose $X$ has type II reduction. Then the boundary component corresponds to the parabolic subgroup that is the stabiliser of a two dimensional isotropic subspace $I_{\Q}$. Let  $P_0$ be the  parabolic subgroup and $I = \langle{e_1,e_2}\rangle$. Now using the computation of the previous section, we know that the weight filtration $H$ is the following:
$$0\subset W_{-2}H= \im(e_1e_2)\subset W_{-1}H = \im(\wedge^1I)\subset W_0H =H$$

Further we have a splitting of the filtration $H_{-2}\oplus H_{-1}\oplus H_0$ such that $H_{-i}$ has a weight $i$-Hodge structure for $0\leq i\leq 2$.

\begin{lem}\label{dimiso}
The dimension of $W_{-2}H$ over $\Q$ is $d/2$ where $2d = \dim H.$
\end{lem}

\begin{proof}
Let $e_3$ be such that $(e_1, e_3) = 1$. Let $U = \langle{e_1,e_3}\rangle$. Similarly let $e_4\in U^{\perp}$ be such that $(e_2,e_4) = 1$. Let $U_2 = \langle{e_2,e_4}\rangle$. Then $(V,q) =(U_1,q)\oplus (U_2,q)\oplus (V^{\prime},q)$ and $Cl(V) = Cl(U_1)\otimes Cl(U_2)\otimes Cl(V^{\prime})$. Since $e_1$ and $e_2$ are isotropic, we see $$e_1e_2Cl(V) = \langle{e_1e_2, e_1e_2e_3, e_1e_2e_4, e_1e_2e_3e_4}\rangle\otimes Cl(V^{\prime}).$$ This proves the lemma.
\end{proof}
\begin{prop}
    The pure Hodge structure  on $\gr_1(H) = H_{-1}$ of weight one corresponds to a product of $d/2$ copies of isogenous elliptic curves where $2d= \dim H.$
\end{prop}

\begin{proof}
 We know that $\im(\wedge^1 I) = H_{-2}\oplus H_{-1}$. Since $H_{-2} = \im(e_1e_2)$, therefore $H_{-1} = \langle{e_1,e_2}\rangle\otimes Cl(I_0\oplus I_{1})$. Hence the Hodge structure on $H_{-1}$ is the direct sum of the weight one Hodge structure on $I = \langle{e_1,e_2}\rangle$. Note that the polarization $\psi$ in Section \ref{KS}, induces a polarization on $H_{-1}$ and hence this polarised weight one Hodge structure corresponds to a product of elliptic curves.

\end{proof}

\subsubsection{Weight filtration on the zero dimensional boundary} For the zero dimensional boundary, we work with the following weight filtration $W_{\bullet}V$:
$$0\subset I = W_{-2}V = W_{-1}V\subset W_{0}V = I^{\perp}\subset W_1V = V$$ 
We can find a co-character $\mu_0:
\bbG_m\rightarrow \SO(V)$ that provides a splitting 
$$I_{-2}\oplus I_0\oplus I_2$$
Now in order to find a lift of this co-character, the calculation in the previous section goes through with slightly different indices. 
\subsubsection{Type III reduction}
In this case $X$ reduces the boundary component that corresponds to the stabiliser of a one dimensional isotropic subspace $I_{\Q}$. Let  $P_0$ be the   parabolic subgroup that stabilises $I = \langle{e_1}\rangle$. Using the computation of the previous section, we know that the weight filtration $H$ is the following:
$$0\subset W_{-2}H =  W_{-1}H= \im(e_1)\subset W_0H =H$$
In this case, the isotropic subspace $W_{-1}H$ has dimension $d$ using a similar argument as done in Lemma \ref{dimiso}. Hence its perpendicular is itself and $\gr_1(H)$ is trivial. That is, the abelian quotient of the Raynaud extension in this case is trivial. 
\subsection{Raynaud extension in characteristic $p$.}

The boundary component  $\mathcal{S}_{K_{\Phi}}(Q_{\Phi}, \mathcal{D}_{\Phi})$ of the Siegal Shimura variety over $\C$ parametrizes a semi-abelian scheme $\mathcal{Z}_\C$ along with a lattice, which we call the universal Raynaud extension over $\C$. By work of \cite{keerthitor} we get an algebraic construction of the boundary components and get the universal Raynaud extension $\mathcal{Z}$ over the canonical integral model   $\mathcal{S}_{K_{\Phi}}(Q_{\Phi}, \mathcal{D}_{\Phi})$, extending its model over $\Q$.  We denote by $\mathcal{Z}_{\F_p}$ its mod $p$ reduction. Let $A$ be an abelian variety over $K$ with bad reduction. Then $A$ is a $K$-point of the formal completion of the connected component of the (compactified) Siegal  moduli space containing $A$ along the toroidal boundary, denoted by $\widehat{\mathcal{A}_{g, K^{\dagger}}^{\text{tor}}}$.  Let $(Z,M)$ be the pull back of $\mathcal{Z}_{\F_p}$ over the $K$-point. Then we have $A\simeq Z/M$ over $K$. See for instance \cite{caraianiraynaud} Section 2.5, and \cite{keerthitor} Section 2.2. In particular, consider Kuga-Satake abelian variety $\KS(X)$ with semi-abelian reduction, that is a $K$-point of the formal completion $\widehat{\mathcal{M}_{K^{\sharp}}^{\text{tor}}}$ along the toroidal boundary stratum the point reduces to. Then the following diagram between the boundary components of $\GSpin$-Shimura variety and Siegal Shimura variety commutes over $\Z_p$ (and hence mod $p)$. 

\[ \begin{tikzcd}
\widehat{\mathcal{M}_{K^{\sharp}}^{\text{tor}}} \arrow{r} \arrow[swap]{d}{} &  \widehat{\mathcal{A}_{g, K^{\dagger}}^{\text{tor}}} \arrow{d}{}& \KS(X)\simeq Z_{\KS}/M\arrow{d} \\%
\mathcal{S}_{\overline{K^{\sharp}}_{\Phi^{\prime}}}(\overline{Q^{\prime}}_{\Phi^{\prime}}, \overline{X_{\Phi^{\prime}}}) \arrow[swap]{d}{} & \mathcal{S}_{\overline{K}^\dagger_{\Phi}}(\overline{Q}_{\Phi}, \overline{\mathcal{D}_{\Phi}}) \arrow{d}{} & Z_{\KS}\arrow{d}\\
\mathcal{A}_1 \arrow[hookrightarrow]{r}& \mathcal{A}_{d/2}&  E^{d/2}\\
\end{tikzcd}
\]
From the calculations in the previous section and with the above discussion, we deduce the following result about the Raynaud extensions of the Kuga-Satake abelian varieties in characteristic $p$. 
\begin{cor}\label{prodec}
Let $X$ be a $K$-point of an orthogonal Shimura variety and let $\KS(X)$ be the associated Kuga-Satake abelian variety of dimension $d$.
\begin{enumerate}
\item Suppose $X$ has type II reduction. Then $\KS(X)$ has semi-abelian reduction. The Raynaud extension $Z_{\KS}$ of the Kuga-Satake abelian variety over $K$ is an extension
$$0\rightarrow \bbG_m^{d/2}\rightarrow Z_{\KS}\rightarrow B\rightarrow 0$$
where $B$ is isogenous to a product of $d/2$ copies of an elliptic curve $E$ over $K$.
\item Suppose $X$ has type III reduction. Then the associated Kuga-Satake abelian variety has totally bad reduction, that is, it degenerates into a torus.  
\end{enumerate}
\end{cor}
\section{Monodromy of the Kuga-Satake abelian variety}\label{10}
In this section we compute the monodromy of the Kuga-Satake abelian variety with semi-stable reduction. The proof is similar to that done in Section \ref{6} for abelian surfaces. Consider $X$ to be an ordinary $K$-point of an orthogonal Shimura variety with bad reduction. Moreover we assume that $\KS(X)$ does not have an isogeny factory of an elliptic curve with semi-stable reduction, in which case we can argue separately by considering the Galois action on the torsion of the two factors separately. 
\subsubsection{When $\KS(X)$ has semi-abelian reduction:}
We know that the Raynaud extension of $\KS(X)$ up to a finite extension fits in an exact sequence as follows:
$$0\rightarrow \mathbb{G}_m^{d/2}\rightarrow Z_{\KS}\rightarrow B\rightarrow 0$$
where $B$ is isogenous to $E^{d/2}$ for some elliptic curve $E$. By Lemma \ref{ordks}, since $\KS(X)$ is ordinary, the elliptic curve $E$ has to be ordinary as well. Furthermore, $\KS(X) \simeq Z_{\KS}/M$ where $M$ is a lattice of rank $d/2$
 in $Z_{\KS}$. Suppose that the lattice is generated by $\langle{m_1,m_2,\dots , m_{d/2}}\rangle$. Then the bases of the $p^n$-torsion of $\KS(X)$ is given by  $y_1^{(n)},y_2^{(n)}, \dots, y_{d/2}^{(n)}$ that generate $Z_{\KS}[p^n](\overline K)$ and points $z^{(n)}_1,\dots, z^{(n)}_{d/2}$ such that $p^n z^{(n)}_i = m_i$ for $1\leq i\leq d/2.$ 

\begin{lem}\label{injtor}
The bases elements $y_1^{(n)},y_2^{(n)}, \dots, y_{d/2}^{(n)}$ of $Z_{\KS}[p^n](\overline K)$ map to the bases of the $p$-power torsion of $E^{d/2}$ for all $n\geq 1$ under the morphism $Z_{\KS}\rightarrow E^{d/2}$. Moreover, the map is injective and Galois invariant on the entire $p$-power torsion of $Z_{\KS}$.
\end{lem}
 
\begin{proof}
We denote the map $Z_{\KS}\rightarrow E^{d/2}$ by $\varphi$. We prove this for $n=1$ and denote $y_i^{(1)} =y_i$ for the ease of notation. The proof for $n>1$ is the same. Suppose that $\sum_{i=1}^{d/2}a_i\varphi ( y_i) = 0$ where $a_i\in\Z/p\Z$ are not all zero. Then $\sum_{i=1}^{d/2}a_iy_i$ lies in $\mathbb{G}_{m}^{d/2}(\overline{K})$ is a $p$-torsion element. Since $\mathbb{G}_{m}^{d/2}[p](\overline{K})=0$, the sum 
$\sum_{i=1}^{d/2} a_iy_i=0$, but since $y_i$ form a bases of $Z_{\KS}[p](\overline{K})$, $a_i =0$ for all $1\leq i\leq d/2.$ This proves that the set of elements $\{\varphi(y_i)\}_{i=1}^{d/2}$  form a bases of the $p$-torsion of $E^{d/2}$.  We prove the rest of the statement for $n=1$ as well and write $z^{(1)}_i = z_i$. Suppose $\sum_{i=1}^{d/2}a_i\varphi(z_i) =0$ such that not all $a_i\in \Z/p\Z$ are zero.  Then $p \sum_{i=1}^{d/2}a_i\varphi(z_i) =0$. That is $\sum_{i=1}^{d/2} a_im_i \in M$ lies in $\bbG_m^{d/2}(\overline{K})$. But then we will have a Tate curve in $Z_{\KS}/M\simeq \KS(X)$, contrary to our assumption. 
 \end{proof}
It is enough to understand the action on the images of the torsion under $\varphi$. The Galois action on the $p^n$-torsion is checked similarly to the case of abelian surfaces and is summarized in the following result.

\begin{lem}\label{galactionks}
 Let $K^{\prime}$ denote the extension of $K$ generated by attaching the $p$-power torsion of $Z_{\KS}(\overline{K})$. Then $Z_{\KS}[p^n](\overline{K})$ is preserved by the Galois group. Moreover, for all $1\leq i\leq d/2$ and $n\geq 1$, we have $\sigma (z^{(n)}_i) - z^{(n)}_i \in Z_{\KS}[p^n](\overline{K})$.  Therefore, the Galois group $\Gal (K^{\prime}(z^{(n)}_1,\dots, z^{(n)}_{d/2} )/K^{\prime})$ is a subgroup of $(\Z/p^n\Z)^{d/2}$.
\end{lem}

 Lemma \ref{galactionks} implies that any element $\sigma$ of $\Gal(K^{\sep}/K)$ acts as the following automorphism on the bases elements of the $p$-power torsion of $A$ as follows:
\[
  \sigma = \left[\begin{array}{ c | c }
   D & W \\
    \hline
    0 & I
  \end{array}\right]
\]
where each block is a $d/2\times d/2$ matrix. The block matrix $I$ is the identity. The block $D$ is a diagonal matrix that determines the action on (the image of) $p^n$-torsion of $Z_{\KS}(\overline K)$, while the right half of the matrix comes from the action on (the image of) $z_i^{(n)}.$ 
 \subsubsection{The case when $B$ has ordinary reduction}\label{ordredks} When $B\simeq E^{d/2}$ has ordinary reduction, then the reduction map $\mathcal{E}^{d/2}(\overline{R})\rightarrow E^{d/2}_0(\overline{\F_q})$ is injective on the $p^n$-torsion of $\mathcal{E}^{d/2}(\overline{R})$ where $\mathcal{E}$ is the N\'eron model of $E$. Hence $K^{\prime}/K$ is an unramified extension. Moreover we claim that the field extension obtained by attaching $z_i^{(n)}$ to $K^{\prime}$, is a ramified extension. Let $K(z^{(n)}_i, y_i^{(n)})$ be the field extension of $K$ obtained by adjoining $y_i^{(n)}$ and $z_i^{(n)}$ for all $1\leq i\leq d/2$. We prove this by showing that the maximal abelian quotient of $\Gal(K(z^{(n)}_i, y_i^{(n)})/K)$ is $\Gal(K^{\prime}/K).$ An argument same as that of Lemma \ref{totram} gives us a contradiction as follows: if $K^{\prime}(z^{(n)}_i)/K^{\prime}$ were unramified, then $K(z^{(n)}_i)$ would be an abelian extension of $K$. The composite of $K^{\prime}$ and $K(z^{(n)}_i)$ would be an abelian extension of $K$ strictly containing $K^{\prime}$, a contradiction. In order to find the maximal abelian extension of $K,$ we compute the commutator of $\Gal(K(z^{(n)}_i,y^{(n)}_i)/K)$ explicitly. From the description of the elements of the Galois group given above, this turns out to be a $d\times d$ block matrix
\[
  \sigma = \left[\begin{array}{ c | c }
    I & W \\
    \hline
    0 & I
  \end{array}\right]
\]
which is exactly equal to $\Gal(K^{\prime}(z^{(n)}_i)/K^{\prime})$.  This completes the proof of the fact that the action of the inertia is unipotent.

  \subsubsection{The case when $B$ has supersingular reduction} We prove an analogue to Igusa's result in this setting.

 \begin{prop}

Consider a sequence of points  of $Z_{\KS}(\overline{K})$: 
$$y^{(n)}_1, y^{(n)}_2, y^{(n)}_3\dots \text{ and } z^{(n)}_1, z^{(n)}_2, z^{(n)}_3 \dots $$ such that $py^{(1)}_i = e$, $py^{(n)}_i = y^{(n-1)}_i$, $pz^{(1)}_i = m_i$ and  $pz^{(n)}_i = z^{(n-1)}_i$  for $1\leq i\leq d/2$ and $n\geq 1$. Then there exists some $n_0$ such that for all $n> n_0$ the field extension  $K(z^{(n)}_i,y^{(n)}_i)$ is a totally ramified extension of $K(z^{(n_0)}_i,y^{(n_0)}_i)$

\end{prop}

 \begin{proof}
It is enough to prove this for the images of the sequence of points via the map $\varphi: Z_{\KS}\rightarrow E^{d/2}$. We proceed similarly as in the case of abelian surfaces. The elliptic curve $E$ over $K$ is ordinary and its special fibre is supersingular. Hence  
by Theorem \ref{ss}, we see that there exists some $n_0$ such that the field generated by attaching the $p^n$-torsion of $E^{d/2}$ denoted by $K(y^{(n)}_i) = K^{\prime}$ is a totally ramified extension of $K(y^{(n_0)}_i)$ for all $n> n_0.$ It remains to show that once we further adjoin $z^{(n)}_i$ to $K^{\prime}$ for $1\leq i\leq d/2$ then, we get a totally ramified extension. This argument is the same as done in Section \ref{ordredks}.
\end{proof}
\subsubsection{Totally bad reduction}
In the case when the Kuga-Satake abelian variety has totally bad reduction, we follow the same argument in Section \ref{badred} for the monodromy representation associated to totally degenerating abelian varieties. 
\section{Monodromy result for orthogonal Shimura varieties in characteristic $p$}\label{k3}
We summarise the calculations for the Kuga-Satake abelian variety in the previous section along with the monodromy result for K3 surfaces in the following main result of our paper.  
\begin{thm}\label{Gspinmainthm}
 Let $X$ be an ordinary $K$-point of the orthogonal Shimura variety associated to the quadratic lattice of signature $(n,2)$.  Let that $\KS(X)$ be its Kuga-Satake abelian variety of dimension $d$. Let $$\rho_{X}: \Gal(K^{\sep}/K)\rightarrow \SO_{n}(\Z_p)$$
$$\rho_{\KS(X)}: \Gal(K^{\sep}/K)\rightarrow \GL_d(\Z_p)$$
 be the associated monodromy representation of $X$ and $\KS(X)$ respectively.
\begin{enumerate}
\item Suppose $X$ has type II reduction such that KS(X) has semi-abelian reduction. 
\begin{enumerate}

\item If $X$ has type II ordinary reduction then the inertia subgroup has unipotent image under $\rho_X$ and $\rho_{\KS}$.
\item If $X$ has type II supersingular then the image of the inertia subgroup has finite index in the entire image of the Galois group $\rho_X$ and $\rho_{\KS}$.
\end{enumerate}
\item Suppose that $X$ has type III reduction, then the associated Galois representations have trivial image.
\end{enumerate}

\end{thm}
\begin{proof}

The monodromy group (image of $\rho_X)$ of the K3 surface is a quotient of the monodromy of the Kuga-Satake abelian variety. From the calculation in the previous section, we get the monodromy results for $\KS(X)$ and hence, for $X$  as well.
\end{proof}

\begin{cor}[Reduction of the Hecke-orbit]\label{finitehecke}
Let $X$ be an ordinary K3 surface with type II supersingular reduction and $\KS(X)$ be its associated Kuga-Satake abelian variety. Then the reduction of their $p$-power Hecke orbit of $X$ is finite.
\end{cor}
We note that in the type II reduction case, the Hecke correspondences extend to the toroidal boundary and hence it makes sense to talk about the reduction of the Hecke orbit. 
\begin{proof}
Since $X$ is ordinary, the Hecke action is compatible with the Kuga-Satake construction.  Hence it is enough to prove this for the associated Kuga-Satake abelian variety. In that case, by replacing $K$ with a finite extension, we may assume that the $p$-power torsion subgroups and hence any $p$-power isogenies  are defined over a  ramified extension of $K$. Thus, the reduction of the Hecke orbit of the Kuga-Satake abelian variety is defined over the residue field $\F_{q^{\prime}}$. Since there are finitely many isomorphism classes of abelian varieties over $\F_{q^{\prime}}$, the result follows.
\end{proof}
 \bibliography{bibli}
\bibliographystyle{amsalpha}

 \end{document}